\documentclass[american]{amsart}

\usepackage[T1]{fontenc}
\usepackage[latin1]{inputenc}
\usepackage[english]{babel}

\usepackage{amsfonts}
\usepackage{amssymb}
\usepackage{amsthm}
\usepackage{amsmath}
\usepackage{mathrsfs}
\usepackage[all]{xy}
\usepackage{graphicx,psfrag,mathrsfs}
\usepackage{hyperref}
\usepackage{svgcolor}

\everymath{\displaystyle}

\newtheorem{theo}{Theorem}

\newtheorem{corollaire}[theo]{Corollary}
\newtheorem{lemme}[theo]{Lemma}

\newtheorem{prop}[theo]{Property}
\newtheorem{props}[theo]{Properties}
\theoremstyle{definition}
\newtheorem{definition}[theo]{Definition}
\newtheorem{def-theo}[theo]{Definition-Property}
\theoremstyle{remark}
\newtheorem{remark}[theo]{Remark}

\newcommand{\sS}{\mathcal{S}\!}

\DeclareMathOperator{\mml}{\mathcal{M}\mathcal{L}}
\DeclareMathOperator{\pmml}{\mathcal{P}\mathcal{M}\mathcal{L}}
\DeclareMathOperator{\gl2}{\mathrm{G}\mathrm{L}_2 \!\left( \mathbb{Z}\right)}
\DeclareMathOperator{\umml}{\mathcal{U}\mathcal{M}\mathcal{L}}

\DeclareMathOperator{\emcg}{\mathrm{MCG}^{e}}

\newcommand{\quot}[2]{{\left. \raisebox{1.5px}{$#1$}\middle/ \raisebox{-2px}{$#2$}\right.}}

\newcommand{\autstr}[2]{\mathrm{Aut}\!\left( #1 , #2 \right)}
\DeclareMathOperator{\wscc}{\mathcal{W}\mathcal{S}}

\renewcommand{\epsilon}{\mathcal{E}}

\begin{document}

\title[A stratification of (projective) measured lamination space]{On a stratification of the space of (projective)  measured laminations}
\author{Vincent Alberge}
%\thanks{ The author is  partially supported by the U.S. National Science Foundation grants DMS 1107452, 1107263, 1107367 ``RNMS: GEometric structures And Representation varieties'' (the GEAR Network).}

\begin{abstract}
We introduce a natural  stratification of the space of projective classes of measured  laminations on a complete hyperbolic surface of finite area. We prove a rigidity result, namely,   the group of self-homeomorphisms of  the space of projective measured  laminations that preserve such a stratification is in general identified with the extended mapping class group of the corresponding surface. We use this approach to fill a gap in the proof of the rigidity of the action of the extended mapping class group on the unmeasured laminations space.
\end{abstract}

\subjclass[2010]{37E30, 57M99.}
\keywords{Measured lamination, stratified space, mapping class group.}

\maketitle

\section{Introduction and notation}

Let $S$ be a complete hyperbolic surface of finite area of genus $g\geq 0$ with $n\geq 0$ cusps. Let us assume further that $(g,n)\neq (0,3), (0,4), (1,1)$.   A \emph{geodesic lamination} on  $S$  is a compact set of $S$ consisting of a union of disjoint simple complete geodesics on $S$ that are called the \emph{leaves} of the geodesic lamination.  A geodesic lamination is  said to be \emph{minimal} if each leaf is dense in the lamination. It is well known that any geodesic lamination is uniquely decomposed into the disjoint union of finitely minimal geodesic laminations together with a finite set of geodesics, each end of which spirals onto a minimal geodesic lamination. A geodesic lamination  is a \emph{measured lamination} on $S$ if there exists a transverse measure $\lambda$ of  support  the lamination itself which is invariant under the translations along the leaves. To avoid heavy notation we denote by  $\lambda$ the measured lamination and by $\left|\lambda\right|$ the support of the transverse measure.   A measured lamination is said to be \emph{minimal} if its support is a minimal geodesic lamination, and a measured lamination is said to be  \emph{uniquely ergodic} if  it admits, up to a multiplication by a positive real number, a unique transverse measure of full support.  It implies that any uniquely ergodic lamination is minimal. A measured lamination is said to be \emph{maximal} if its support is not properly contained in the support of a measured lamination. The set of measured laminations on $S$ is denoted by  $\mml$. We  assume that the empty lamination $\emptyset$ is a measured lamination. Elementary non-trivial examples of measured laminations are given by the set of simple closed geodesics denoted in this paper by  $\mathcal{S}$. The transverse measure associated with a simple closed geodesic  is the counting measure, namely, the measure that assigns to any geodesic transverse arc its number of intersection points with such a simple closed  geodesic. %A simple closed geodesic $\alpha$, when seen as an element of $\mml$, will be also denoted by $\alpha$. 
Other examples of measured laminations are given by the set  of \emph{weighted} simple closed geodesics  $\wscc$ of $S$, that is, the set of measured laminations of support a simple closed geodesic on $S$ and where the transverse measure is the corresponding counting measure multiplied by a positive real number. Let us point out that elements of $\wscc$ are  uniquely ergodic measured laminations.

We now recall that for any element $\lambda$ of $\mml$ and any element $\alpha\in\sS$, we can define the \emph{geometric intersection number} $i\left(\lambda, \alpha \right)$ as the measure of $\alpha$ with respect to  the measure $\lambda$. The geometric intersection number   induces a one-to-one mapping from $\mml$ to the space $\mathbb{R}_{+}^{\sS}$ of nonnegative functions  on $\sS$. We then equip the space $\mml$ with the topology defined by  pointwise convergence on $\sS$. It is a well-known theorem of Thurston  that the space $\mml$, when endowed with such a topology, is homeomorphic to an open ball of dimension $6g-6+2n$ and the set $\wscc$ is a dense subset in $\mml$. Furthermore,  it is also known that the intersection number defines a continuous function $i\left( \cdot , \cdot \right)$ from $\mml \times \mml$ to $\mathbb{R}$.
%Furthermore, although it is obvious from what is written above let us add that if $\alpha$ and $\beta$ are two simple closed geodesics then $i(\alpha, \beta)$ is the number of intersection points between $\alpha$ and $\beta$. 
 
Let us continue to introduce notations. For a minimal measured lamination $\lambda$ which is not a weighted simple closed geodesic, there exists a unique connected subsurface of $S$ with totally geodesic boundary $\Sigma\left( \gamma \right)$, called the \emph{supporting surface} of $\lambda$, such that for any simple closed geodesic $\alpha$ in the interior of  $\Sigma\left( \gamma \right)$, we have $i(\lambda, \alpha)\neq 0$. Any connected component of the boundary  of $\Sigma\left( \gamma \right)$ is called a \emph{peripheral curve} associated with $\lambda$. In particular, if $\alpha$ is a peripheral curve associated with $\lambda$, then  for any measured lamination $\mu$ such that $i(\alpha, \mu)\neq 0$, $i(\lambda, \mu)\neq 0$. And conversely, such  a property characterises peripheral curves.

The natural action of the set of  positive real numbers $\mathbb{R}_+^{*}$ on $\mml$ which consists of multiplying the transverse measure by a constant positive number induces the coset $\quot{\mml\setminus\left\lbrace \emptyset \right\rbrace}{\mathbb{R}_+^{*}}$, denoted by $\pmml$ and called the space of \emph{projective measured laminations} of $S$.   The space $\pmml$ endowed with the quotient topology is therefore homeomorphic to the unit sphere $\mathbb{S}^{6g-7+2n}$ and admits the set $\sS$ as a dense subset.%Let us also point out that such a space can be seen as a subset of $\textrm{P}\mathbb{R}_{+}^{\sS}=\quot{\mathbb{R}_{+}^{\sS}\setminus\left\lbrace 0 \right\rbrace}{\mathbb{R}_+^{*}}$. 
We  denote the projective class of any element $\lambda\in\mml$ by $\left[\lambda\right]$. The support of a projective measured lamination is the support of any measured lamination which belongs to such an equivalence class.

For more details about geodesic laminations and measured laminations  we refer to \cite{bonahon,casson&bleiler} and \cite{canary&epstein&marden}. For another point of view which is more topological and which deals with the notion of measured foliations instead, we can refer to  \cite{flp}.

Let us also introduce the so-called \emph{unmeasured lamination space} $\umml$. Such a space is the quotient space of $\mml\setminus\left\lbrace \emptyset \right\rbrace$ (or $\pmml$) obtained by forgetting the transverse measures. Therefore, if $\lambda\in\mml\setminus\left\lbrace \emptyset\right\rbrace$ then its support $\left| \lambda\right|$ can be identified with its image in $\umml$. One can also think the set $\mathcal{S}$ as a subset of $\umml$. When equipped with the quotient topology, $\umml$ is a non-Hausdorff space. For further details about such a space and its  topology, we refer to \cite{papadop-uml}, \cite{papadop2} and \cite{ohshika1}. Although it is not common, when the topology is not involved  we  see $\umml$ as a subset of the set of geodesic laminations by identifying  its elements with their corresponding supports.

There exists a group that acts in a natural way on all of the spaces defined above. Such a group is the so-called \emph{extended mapping class group} of $S$ which is defined as the group of isotopy classes of all homeomorphisms of $S$. We denote it by  $\emcg$.

In this paper, after defining the notion of stratification in Section \ref{sec:def}, we shall see in Section \ref{sec:strat}  that there exist a stratification $\mathscr{S}$ of $\mml$ and a stratification $\mathrm{P}\!\mathscr{S}$ of $\pmml$ that are induced by the  unmeasured lamination space. In  Section \ref{sec:rigidity}, we shall prove the main result of this paper, namely, the group of self-homeomorphisms of $\pmml$ that preserve the stratification $\mathrm{P}\!\mathscr{S}$ is induced by $\emcg$.  This rigidity result for the action of the (extended) mapping class group is in the lineage of previous results proved by various authors, such as \cite{charitos&papadopoulos&papadoperakis,papadop2, ohshika,alberge&miyachi&ohshika} and most recently \cite{papadop&ohshika1} and \cite{papadop&ohshika2}. This list being not all-inclusive, we refer to the survey \cite{papadop3}. What is in common from all these papers is the use of results from \cite{ivanov,korkmaz} and \cite{luo}  saying that the group of automorphisms of the complex of curves of $S$ is (most of the time) induced by the extended mapping class group. Therefore, the key point in our case is to prove that any self-homeomorphism of $\pmml$ that preserves the stratification $\mathrm{P}\!\mathscr{S}$ defines an automorphism of the complex of curves. In the last section of this paper we shall fix a gap in the proof  on the rigidity for the action of the (extended) mapping class group  on $\umml$ that appeared in \cite{papadop2}. %See Lemma \ref{lemma:key} below.

%Let us end this section by recalling that the \emph{complex of curves} of $S$, denoted by $\ccs$, is a simplicial complex such that for any  $k\geq 0$, the $k$ simplices are $k+1$ pairwise disjoint simple closed geodesics. 

\section{Stratified space}\label{sec:def}

Most of the notions introduced in this section are based on \cite{kloeckner}. Let $X$ be a Hausdorff topological space, let $A$ be a set, and let $\mathscr{S}=\left\lbrace \mathscr{X}_i \right\rbrace_{i\in A}$ a  partition of $X$ into locally closed sets. We recall that a subset of $X$ is locally closed if it is open in its closure.

\begin{definition}\label{def:1}
We say that $\mathscr{S}$ is a stratification of $X$ if 
$$
\forall i, j \in A,\; \mathscr{X}_i \cap \overline{\mathscr{X}_j} \neq \emptyset \iff \mathscr{X}_i \subset \overline{\mathscr{X}_j}.
$$
In this case, each element $\mathscr{X}_i$ of the partition is called a stratum  and the pair $\left( X, \mathscr{S}\right)$ is what we call a stratified space.
\end{definition}

Throughout the rest of this section, we assume that $\left(X, \mathscr{S}\right)$ is a stratified space. Let us now define a partial order $\preceq$ on the set $A$ as follows. Let $i$ and $j$ be two elements of $A$. We say that $i\preceq j$ if $\mathscr{X}_i \subset \overline{\mathscr{X}_j}$. Furthermore, for $i,j\in A$, we set $i \prec j$ if $i\preceq j$ and $i\neq j$. This partial order on $A$ leads to the notion of depth.

\begin{definition}
The depth of the stratum $\mathscr{X}_i$ (for $i\in A$), denoted by $\mathfrak{d}(\mathscr{X}_i)$, is the supremum over the length $m$ of sequences $(i_0,i_1,\cdots i_m)$ such that 
$$
i=i_0  \succ i_1 \succ \cdots \succ i_m .
$$

Moreover, we define the depth of $\left( X, \mathscr{S}\right)$ as 
$$
\mathfrak{d}\!\left( \left( X, \mathscr{S}\right)\right) = \sup\left\lbrace \mathfrak{d}(\mathscr{X}_i)\mid i \in A \right\rbrace.
$$ 
\end{definition}

Let us point out that the depth of a stratum, and therefore the depth of $\left( X, \mathscr{S}\right)$, might be infinite. However, in the stratifications studied in this paper, the corresponding depths will be finite.

We also define the depth of an element $x\in X$ as the depth of the stratum that contains it and we denote it by $\mathfrak{d}\!\left(x\right)$.

We continue this section with the following definition.
\begin{definition}
We say that a homeomorphism $f:X \rightarrow X$ is an automorphism of the stratified space $\left( X, \mathscr{S}\right)$ if $f$ maps each stratum to a stratum. We denote by $\autstr{X}{\mathscr{S}}$ the group of automorphisms of $\left( X,\mathscr{S}\right)$.
\end{definition}

Elementary examples of stratified spaces are given by  CW-complexes. In particular, any simplicial complex can be viewed as a stratified space.

\section{Stratifications and preliminary results}\label{sec:strat}

As for the decomposition of geodesic laminations into minimal geodesic laminations we recall that any measured lamination $\lambda$ is uniquely decomposed into minimal components as follows:
\begin{equation}\label{eq:decomp}
\lambda =  \sum_{i}\lambda_i +\sum_{j}\alpha_j +\sum_{k}\beta_k;
\end{equation}
where $\lambda_i$ is a minimal measured lamination contained in $\lambda$ which is not a weighted simple closed geodesic, $\alpha_j$ is a peripheral curve associated with some $\lambda_i$ which is contained in  $\lambda$, and $\beta_k$ is a weighted simple closed geodesic  contained in $\lambda$ which is not a peripheral curve.  We used sums in (\ref{eq:decomp}) in the sense that 
$$
\forall \mu\in\mml,\; i\left(\lambda, \mu\right) =  \sum_{i}i\left(\lambda_i, \mu\right) +\sum_{j}i\left(\alpha_j,\mu\right) +\sum_{k}i\left(\beta_k,\mu\right).
$$ 
Therefore, because of topological constraints any measured lamination admits at most $3g-3 +n$ minimal components. By abuse of notation, for two projective measured laminations $\left[\lambda\right]$ and $\left[ \mu\right]$ of disjoint supports, we set $\left| \left[ \lambda \right] + \left[ \mu \right] \right|$ for the (unmeasured) geodesic lamination which is the union of the support of $\left[ \lambda\right]$ with the support of $\left[ \mu\right]$.

We are now ready to introduce a notion of stratification of $\mml$. For any $\lambda\in\mml$, we set 
$$
\mathscr{C}_{\left| \lambda \right|}=\left\lbrace \mu\in\mml  \mid \left| \mu \right| = \left| \lambda \right| \right\rbrace,
$$
and we  call it the \emph{cone} associated  with $\lambda$. It is  known that such a cone is a topological manifold of dimension  at most $3g-3+n$. See \cite{katok,plante,levitt}, and \cite{papadop1}.  Furthermore, it is obvious that the collection of sets $\mathscr{S}=\left\lbrace \mathscr{C}_{\left| \lambda\right|}\right\rbrace_{\left|\lambda\right|\in\umml}$ realises a partition of $\mml$. 

From the definition of a cone associated with a measured  lamination we obtain the following lemma.
\begin{lemme}
Let $\lambda$ and $\mu$ be two measured  laminations. The following properties are equivalent:
\begin{enumerate}
\item $\left| \lambda \right| \subset \left| \mu\right|$,
\item $\mathscr{C}_{\left|\lambda\right|}\bigcap \overline{\mathscr{C}_{\left|\mu\right|}}\neq \emptyset$,
\item $\mathscr{C}_{\left|\lambda\right|} \subset \overline{\mathscr{C}_{\left|\mu\right|}}$.
\end{enumerate}
\end{lemme}
\begin{proof}
This is just a consequence of the uniqueness of the decomposition into minimal components. \end{proof}
The lemma above implies that  the partition $\mathscr{S}$ of $\mml$  defines a stratification of $\mml$ in the sense of Definition \ref{def:1}, the cones being the strata and the partial order on $\umml$ being the inclusion. 
 
 Such a stratification induces a stratification of $\pmml$. Indeed, if for $\left[\lambda\right]\in\pmml$ we set 
$$
\mathrm{P}\mathscr{C}_{\left| \lambda \right|}=\left\lbrace \left[\mu\right]\in\pmml  \mid \left| \mu \right| = \left| \lambda \right| \right\rbrace,
$$
then $\mathrm{P}\!\mathscr{S}=\left\lbrace \mathrm{P}\mathscr{C}_{\left| \lambda\right|}\right\rbrace_{\left|\lambda\right|\in\umml}$ is a stratification of $\pmml$.

We now derive the following elementary observations.
\begin{props}\label{props1} Let $\lambda \in \mml$. \begin{enumerate}
\item If $\lambda$ is minimal, then $\mathfrak{d}\!\left(\lambda\right)=\mathfrak{d}\left( \left[\lambda\right] \right)=0$. The converse is also true.
\item The depths $\mathfrak{d}\!\left(\lambda \right)$ and $\mathfrak{d}\!\left(\left[\lambda\right]\right)$ are equal to at most $3g-4+n$. Furthermore, if $\lambda$ is such that the union of its minimal components defines a generalized pair of pants decomposition of $S$, then $\mathfrak{d}\!\left(\lambda\right)=\mathfrak{d}\!\left( \left[\lambda\right] \right)=3g-4+n$.

Thus, $\mathfrak{d}\!\left(\left( \mml, \mathscr{S}\right) \right)=\mathfrak{d}\!\left(\left( \pmml, \mathrm{P}\!\mathscr{S}\right) \right)=3g-4+n$.

\end{enumerate}
\end{props}

Let us recall that a generalized pair of pants is a hyperbolic surface which is homeomorphic to a domain of the extended complex plane  with three boundary components, where each boundary component is either  a point  or a Jordan curve. 

\begin{proof}
The proofs of all of the statements above follow from the fact that if $\lambda$ is decomposed into $k$ minimal components, then $k-1=\mathfrak{d}\!\left( \lambda\right)=\mathfrak{d}\!\left( \left[\lambda\right]\right)$.
\end{proof}

%\begin{remark}
%The number $\mathfrak{d}\!\left(\lambda\right)$ corresponds to the so-called adherence height of $\left| \lambda \right|$ introduced in \cite{ohshika}, whenever $\left| \lambda\right|$ is seen as an element of $\umml$. 
%\end{remark}

Let us point out that it is not because the depth of a measured lamination is equal to $3g-4+2n$ than such a measured lamination defines a generalized pair of pants decomposition of $S$. Indeed, if the genus of $S$ is at least $1$, then  there always exists a minimal measured  lamination $\lambda_0$ in $S$ for which $\Sigma\left(\lambda_0\right)$ is of genus $1$ with one (closed) boundary component. We also say that $\lambda_0$ \emph{fills} a once-punctured torus. Let $\lambda_1$ be the boundary component of $\Sigma\left(\lambda_0\right)$. Let $\delta$ be a simple closed geodesic that belongs to the interior of $\Sigma\left(\lambda_0\right)$.  Now, we pick exactly $3g-5+n$ simple closed geodesics $\lambda_i$ ($2\leq i \leq 3g-4+n$) in such a way that $\left\lbrace \delta, \lambda_1 ,\cdots ,\lambda_{3g-4+n}\right\rbrace$ is a (generalized) pair of pants decomposition of $S$. Therefore, if we set $\lambda = \sum_{i=0}^{3g-5+n}\lambda_i$, then $\mathfrak{d}\!\left( \lambda\right)=3g-4+n$ although it is not a (generalized) pair of pants decomposition of $S$. 

\section{A Rigidity result}\label{sec:rigidity}

It is elementary to observe that any element of $\emcg$ induces an element of $\autstr{\pmml}{\mathrm{P}\mathscr{S}}$, and as mentionned in the introduction, we shall prove in this section the converse, namely,  any element of $\autstr{\pmml}{\mathrm{P}\mathscr{S}}$ is induced by an element of $\emcg$. For this purpose we shall  prove that any element of $\autstr{\pmml}{\mathrm{P}\mathscr{S}}$ induces an automorphism of the corresponding complex of curves. We recall that the \emph{complex of curves} of $S$ is a simplicial complex such that for any  $k\geq 0$, the $k$ simplices are $k+1$ pairwise disjoint simple closed geodesics. 

Let us first prove some preliminary results. 

\begin{lemme}\label{lemme:2}
Let $f\in\autstr{\mml}{\mathscr{S}}$ and let $F\in\autstr{\pmml}{\mathrm{P}\mathscr{S}}$. Then $f$ (resp. $F$) maps any minimal  measured   lamination (resp. minimal projective measured  lamination) to a minimal measured  lamination (resp. minimal projective measured  lamination).
\end{lemme}
\begin{proof}
This is just because such  mappings $f$ and $F$ preserve the depth of  strata.
\end{proof}

\begin{prop}\label{lemme:disjoint}
Let $f\in\autstr{\mml}{\mathscr{S}}$ and let $F\in\autstr{\pmml}{\mathrm{P}\mathscr{S}}$. Let $\lambda$ and $\mu$ be two measured  laminations with disjoint supports. Then, $f\left( \lambda\right)$ (resp. $F\left( \left[\lambda\right]\right)$) and $f\left(\mu\right)$ (resp. $F\left( \left[\mu\right]\right)$) have disjoint supports.
\end{prop}

\begin{proof}
Let $f$, $F$, $\lambda$, and $\mu$ be as in the statement.  Then, we have $\overline{\mathscr{C}_{\left| \lambda\right|}} \bigcap \overline{\mathscr{C}_{\left| \mu\right|}} = \emptyset$. Since $f\left(\mathscr{C}_{\left| \lambda\right|} \right) = \mathscr{C}_{\left| f\left( \lambda\right)\right|}$ and $f\left(\mathscr{C}_{\left| \mu\right|} \right) = \mathscr{C}_{\left| f\left( \mu\right)\right|}$, we deduce that $\overline{\mathscr{C}_{\left| f\left( \delta\right)\right|}}\bigcap \overline{\mathscr{C}_{\left| f\left( \mu\right)\right|}}=\emptyset$, meaning that $f\left(\lambda\right)$ and $f\left( \mu\right)$  have disjoint supports. We prove in the same way that  $F\left( \left[ \lambda\right]\right)$ and $F\left( \left[ \mu\right]\right)$ have disjoint supports.
\end{proof}

Let us see a direct consequence of this property.

\begin{corollaire}\label{cor:disjoint}
Let $f\in\autstr{\mml}{\mathscr{S}}$ and let $F\in\autstr{\pmml}{\mathrm{P}\mathscr{S}}$. Let $\lambda$ and $\mu$ be two measured  laminations with disjoint supports. Then, $f\left( \mathscr{C}_{\left| \lambda + \mu \right|} \right) = \mathscr{C}_{\left| f\left(  \lambda \right)+f\left( \mu\right)\right|}$ and $F\left( \mathrm{P}\mathscr{C}_{\left| \lambda + \mu \right|} \right) = \mathrm{P}\mathscr{C}_{\left| F\left(  \left[\lambda\right] \right)+F\left( \left[\mu\right]\right)\right|}$.
\end{corollaire}

\begin{proof}
Let $f$, $F$, $\lambda$, and $\mu$ be as above. We set $\delta  = \lambda + \mu$. Because of Property \ref{lemme:disjoint}, we can set $\nu = f\left( \lambda\right)+f\left( \mu \right)$. We then have to show that $\left| f\left( \delta\right)\right|= \left| \nu \right|$.  One  has that both cones $\mathscr{C}_{\left| f\left(\lambda\right)\right|}$ and $\mathscr{C}_{\left| f\left(\mu\right)\right|}$ are contained in $\overline{\mathscr{C}_{\left| \nu\right|}}$ and in $\overline{\mathscr{C}_{\left| f\left( \delta\right)\right|}}$. This implies that $\left| f\left( \delta\right) \right|$ contains $\left| f\left(\lambda\right) \right|$ and $\left| f\left( \mu\right) \right|$. Furthermore, because $f$ preserves the depth of strata, $\mathfrak{d}\!\left(  f\left( \delta\right) \right) = \mathfrak{d}\!\left( \nu \right)$. Thus, $\left| f\left( \delta\right)\right|= \left| \nu \right|$. Using the same steps, we prove $F\left( \mathrm{P}\mathscr{C}_{\left| \lambda + \mu \right|} \right) = \mathrm{P}\mathscr{C}_{\left| F\left(  \left[\lambda\right] \right)+F\left( \left[\mu\right]\right)\right|}$.
\end{proof}

Another consequence of Property \ref{lemme:disjoint} is the following:

\begin{corollaire}\label{cor:key}
Let $f\in\autstr{\mml}{\mathscr{S}}$ and let $F\in\autstr{\pmml}{\mathcal{P}\mathscr{S}}$.  Let $\alpha$ be a a weighted simple closed geodesic such that its support is a  peripheral curve associated with some minimal measured lamination $\lambda$. Then, $f\left( \alpha\right)$ (resp. the support of $F\left(\left| \alpha\right| \right)$) is a peripheral curve associated with $f(\lambda)$ (resp. the minimal measured lamination corresponding to $F(\left| \lambda\right| )$).
\end{corollaire}

\begin{proof}
Let $f$, $\alpha$, and $\lambda$ be as in the statement. Let $\mu\in\mml$ be such that $i\left( \alpha, \mu \right)>0$. Then, $\left| f(\alpha) \right| \cap \left| \mu \right| \neq \emptyset$. Therefore, $\left| \alpha \right| \cap \left| f^{-1}(\mu)\right| \neq \emptyset$ because otherwise by Property \ref{lemme:disjoint} we would  have  $\left| f(\alpha) \right| \cap \left| \mu \right| = \emptyset$. Using the fact that $\left| \alpha \right|$ is a peripheral curve, we have $i(\lambda, f^{-1}(\mu))>0$, that is,  $\left| \lambda \right| \cap \left| f^{-1}(\mu)\right| \neq \emptyset$, which, by using once again Property \ref{lemme:disjoint}, implies that $\left| f(\lambda) \right| \cap \left| \mu\right| \neq \emptyset$. This proves that $f(\alpha)$ is a peripheral curve associated with $f(\lambda)$.
\end{proof}

Such a corollary will be the key for proving:

\begin{lemme}\label{lemma:key}
Let $f\in\autstr{\mml}{\mathscr{S}}$ and let $F\in\autstr{\pmml}{\mathcal{P}\mathscr{S}}$. Then, $f\left( \wscc \right) = \wscc$ and $F\left( \mathcal{S} \right) = \mathcal{S} $.
\end{lemme}

\begin{proof}
Let $f$ and $F$ be as in the statement. Since $f$ and $F$ are bijective, it is enough to prove that $f\left( \wscc \right) \subset \wscc$ and $F\left( \mathcal{S} \right) \subset \mathcal{S} $. 

Let $\alpha \in \wscc$. By Lemma \ref{lemme:2}, $f(\alpha)$ is a minimal measured  lamination. Furthermore, $f(\alpha)$ is not maximal because otherwise $\alpha$ would be maximal.  Let us prove that $f(\alpha)$ is a weighted simple closed geodesic. For the sake of contradiction one assumes that it is not a weighted simple closed geodesic. Thus, there exists a weighted simple closed geodesic $\beta$ which corresponds to a peripheral curve associated with $f(\alpha)$. Since $f^{-1}\in\autstr{\mml}{\mathscr{S}}$, by Corollary \ref{cor:key} one can say that the support of $f^{-1}(\beta)$ is a peripheral curve associated with $\alpha$, which is a contradiction since $\alpha$ is a weighted simple closed geodesic. Therefore, $f\left( \wscc\right)\subset \wscc$. 

The proof of $F\left( \mathcal{S}\right)\subset \mathcal{S}$ is identical.
\end{proof}

One can also use the same strategy as in \cite{papadop&ohshika2} in order to prove Corollary \ref{cor:key}. Indeed, let us set for any $\lambda \in \mml$, $\mathcal{U}_{\lambda}=\left\lbrace  \mu\in\mml \mid \left| \lambda \right| \cap \left| \mu \right| = \emptyset\right\rbrace$. Therefore, if $\alpha\in \wscc$ then $\mathcal{U}_{\alpha}$ is homeomorphic to a space of dimension $6g-8+2n$. This is because if $\left| \alpha \right|$ is not a separating curve, that is, if $S\setminus \left| \alpha \right|$ is connected, then  $\mathcal{U}_{\alpha}$ is homeomorphic to the  measured  lamination space of a hyperbolic surface of genus $g-1$ with $n+2$ cusps. On the other hand, if $\alpha$ is a separating curve, then $\mathcal{U}_{\alpha}$ is homeomorphic to the Cartesian product of the measure lamination space of a surface of genus $g_1$  with $n_1 + 1$ cusps and  the measured  lamination space of a surface of genus $g_2$ with  $n_2+1$ cusps such that $g_1+g_2=g$ and $n_1+n_2=n$. Thus, $\mathcal{U}_{\alpha}$ is also of dimension $6g-8+2n$. Furthermore, one has that if $\lambda$ is a minimal measured  lamination which is neither a weighted simple closed geodesic nor a maximal measured lamination, then $\mathcal{U}_{\lambda}$ is homeomorphic to a space  of dimension $6\widetilde{g}-12+2\widetilde{n}+3p$, where $g-\widetilde{g}$, $n-\widetilde{n}$, and $p$ are respectively the genus, the number of cusps, and the number of boundary components of the supporting  surface $\Sigma\left( \lambda \right)$.  Such a dimension is less than $6g-8+2n$.  Therefore, for any $\lambda\in\mml$, one has that $\mathcal{U}_{\lambda}$ is of dimension less than or equal to $6g-8+2n$, and equality holds if and only if $\lambda \in\wscc$. Having this in mind and Property \ref{lemme:disjoint} it becomes elementary to prove that elements of $\autstr{\mml}{\mathscr{S}}$ (resp. $\autstr{\pmml}{\mathcal{P}\mathscr{S}}$) preserve $\wscc$ (resp. $\mathcal{S}$). % From what we observed above we have $\mathcal{U}_{\alpha}$ homeomorphic to a space of dimension $6g-8+2n$. Therefore, since $f$ is a homeomorphism, $f\left( U_{\alpha}\right)$ is also homeomorphic to a space of dimension $6g-8+2n$. Moreover, by Property \ref{lemme:disjoint} we have that $f\left( U_{\alpha}\right)=U_{f\left(\alpha\right)}$, so the support of $f(\alpha)$ must be a simple closed geodesic and therefore $f\left( \wscc\right)\subset \wscc$.  

\begin{remark}\label{rmk:ohshika}
Let us give a third way to prove Lemma \ref{lemma:key} which does not involve the topological dimension of a particular space.  The author wants to thank Ken'ichi Oshika for having pointed out this alternative to him. 

Let us start by assuming  that the genus $g$ of $S$ is greater than $1$. Otherwise the following observation is not working.  If $\lambda$ is a minimal measured  lamination which is neither a weighted simple closed geodesic nor a maximal measured  lamination, then for any two maximal measured  laminations $\mu_1$ and $\mu_2$ containing  $\lambda$ as a sublamination we have 
$$
\left( \left| \mu_1 \right| \setminus \left| \lambda \right| \right) \cap \left( \left| \mu_2 \right| \setminus \left| \lambda \right| \right) \neq \emptyset.
$$
This is obvious since any maximal measured  lamination containing $\lambda$ has to contain also the associated peripheral curve(s). And conversely, if a minimal measured   lamination which is not maximal statisfies such a property, then it cannot be a simple closed geodesic. 

Let us now see why this implies that any element of $\autstr{\mml}{\mathscr{S}}$ preserves the set $\wscc$.  Let $f\in\autstr{\mml}{\mathscr{S}}$ and let $\alpha \in \wscc$.  Let $\mu_1$ and $\mu_2$ be two maximal measured laminations having $\alpha$ as minimal component and such that $\left( \left| \mu_1 \right| \setminus \left| \alpha \right| \right) \cap \left( \left| \mu_2 \right| \setminus \left| \alpha \right| \right) = \emptyset$. By  Property \ref{lemme:disjoint}, $\left( \left| f\left( \mu_1 \right) \right| \setminus \left| f\left(\alpha \right) \right| \right) \cap \left( \left| f\left(\mu_2 \right) \right| \setminus \left| f\left( \alpha \right) \right| \right) = \emptyset$, and since $f$ preserves the depth of strata we have that $f\left( \mu_1 \right)$ and $f\left( \mu_2 \right)$ are two maximal measured laminations which have by Lemma \ref{lemme:2} the measured lamination $f\left( \alpha\right)$ as minimal component. Therefore, $f\left( \alpha \right)$  has to be a weighted simple closed geodesic and we then proved that $f$ preserves $\wscc$. 
\end{remark}

We are now ready to state and prove the main theorem.

\begin{theo}\label{main-theo}
If $(g,n)\neq (2,0), (1,2)$, then $\autstr{\pmml}{\mathrm{P}\!\mathscr{S}} \simeq \emcg$; while if $(g,n)=(2,0)$ or $(g,n)=(1,2)$, then $\autstr{\pmml}{\mathrm{P}\!\mathscr{S}} \simeq \quot{\emcg}{\quot{\mathbb{Z}}{2\mathbb{Z}}}$.
\end{theo}

\begin{proof}
Let us prove that the natural homomorphism from $\emcg$ to $\autstr{\pmml}{\mathrm{P}\!\mathscr{S}}$ is an isomorphism if $(g,n)\neq(2,0), (1,2)$, and is surjective with a kernel of order $2$ otherwise.
Let $F\in\autstr{\pmml}{\mathrm{P}\!\mathscr{S}}$. Because of Lemma \ref{lemma:key}, $F$ preserves the set $\mathcal{S}$ and because $F$ also preserves the depth of elements of $\pmml$ one has that $F$ actually induces  an automorphism of the complex of curves. 

\underline{Case I}: If  $(g,n)\neq (2,0), (1,2)$.  By \cite{ivanov,korkmaz} and \cite{luo} one has that there exists a unique $g\in\emcg$ such that $g=F$ on $\mathcal{S}$, and therefore  by density of $\mathcal{S}$ in $\pmml$, one concludes that $g=F$ on $\pmml$, proving that $\emcg \simeq \autstr{\pmml}{\mathrm{P}\!\mathscr{S}}$.

\underline{Case II}: If $(g,n)=(1,2)$.  Let us prove first that $F$ also preserves the separating curves on $S$. Let $\alpha$ be a  simple closed geodesic which is separating. Then, $S\setminus \alpha$ has two connected components; one of them is isometric to the interior of a generalized hyperbolic pair of pants with two cusps, and the other one is isometric  to the interior of a  one-holed hyperbolic torus. On the latter component one can define a minimal measured lamination $\lambda$ which is not a simple closed geodesic. The peripheral curve associated with $\lambda$ is then $\alpha$. %Since $\lambda$ is uniquely ergodic, by an abuse of language one  identifies $\lambda$ with its projective class. 
Therefore, by Corollary \ref{cor:key}  $F\left(\left[\lambda\right]\right)$ is (the projective class of)  a minimal measured lamination which is not a  simple closed geodesic and  its associated  peripheral curve is $F\left(\left[\alpha\right)\right)$ which is therefore a separating curve. Thus, by \cite{luo} one has that there exists a $g\in\emcg$ such that $g=F$ on $\mathcal{S}$, and by density one can say that the homomorphism from $\emcg$ to $\autstr{\pmml}{\mathrm{P}\!\mathscr{S}}$ is onto. Furthermore, because of the existence of a hyperelliptic involution one has that the kernel of such a homomorphism is of order two. Therefore, $\autstr{\pmml}{\mathrm{P}\!\mathscr{S}} \simeq \quot{\emcg}{\quot{\mathbb{Z}}{2\mathbb{Z}}}$.

\underline{Case III}: If $(g,n)=(2,0)$. As above, one still has that $F$ induces an automorphism of the corresponding complex of curves and therefore there exists $g\in \emcg$ such that $g=F$ on $\mathcal{S}$. Again,  by density of $\mathcal{S}$ one has that $g=F$ on $\pmml$. Because of the existence of a hyperelliptic involution on such a surface one concludes as before that $\autstr{\pmml}{\mathrm{P}\!\mathscr{S}} \simeq \quot{\emcg}{\quot{\mathbb{Z}}{2\mathbb{Z}}}$.
\end{proof}

\begin{remark}
We can also obtain Theorem \ref{main-theo} in the case of closed hyperbolic surfaces in a different way.  Indeed, one can use  \cite[Th\'{e}or\`{e}me 4]{papadop&ohshika2} after  observing that  Property \ref{lemme:disjoint} implies that for any $f\in\autstr{\mml}{\mathscr{S}}$, if $i\left( \lambda , \mu \right) =0$ then $i\left( f\left(\lambda\right) , f\left(\mu\right) \right)=0$. Moreover, one can easily prove that if $f$ is a self-homeomorphism of $\mml$ such that for any pair $(\lambda, \mu)$ of measured  lamination with $i(\lambda, \mu)=0$ we have $i\left( f\left(\lambda\right), f\left(\mu\right)\right)=0$,  then $f\in\autstr{\mml}{\mathscr{S}}$. 
\end{remark}

Let us continue this section by justifying the fact that we started this paper by assuming that $(g,n)\neq (0,3), (0,4), (1,1)$. If $(g,n)=(0,3)$, then $\pmml$ is the empty set and therefore there is no stratification.  Now, let $S$ be  either  a four-times-punctured hyperbolic sphere or a once-punctured hyperbolic torus. It is known that in this case the projective lamination space $\pmml$ is homeomorphic to $\mathbb{R}\cup\left\lbrace \infty\right\rbrace$---the one-point compactification of the real line. This comes from the fact that $\pmml$ is the Thurston compactification of  the Teichm\"{u}ller space of the surface $S$. By means of such a correspondence, the set of simple closed geodesics is identified with $\mathbb{Q}\cup\left\lbrace \infty \right\rbrace$,  whereas the set of uniquely ergodic  laminations which are not  simple closed geodesics is identified with $\mathbb{R}\setminus\mathbb{Q}$. This implies that strata are actually points, because there are all of depths equal to $0$. Therefore, one has the group $\autstr{\pmml}{\mathrm{P}\mathscr{S}}$  isomorphic to the group of self-homeomorphisms of $\mathbb{R}\cup\left\lbrace \infty \right\rbrace$, and  it is well known that such a group is much bigger than $\emcg \simeq  \gl2$.

\section{About the self-homeomorphisms of $\umml$}\label{sec:uml}

While working on this paper, the author observes a gap in the proof of the result of Papadopoulos in \cite{papadop2} which asserts that any self-homeomorphism of $\umml$ (when endowed with the quotient topology) preserves the set of simple closed geodesics. This section is devoted in fixing such a gap.

Before explaining where exactly the gap is, let us recall a few things about the space $\umml$. As mentionned earlier, this space when equipped with the quotient topology is non-Hausdorff. Furthermore, each stratum of the stratification $\mathscr{S}$ of $\mml$ is mapped onto a point in $\umml$. Following the terminology of \cite{ohshika}, one recalls  that $\left| \lambda \right| \in\umml$ is said to be \emph{unilaterally adherent} to $\left|\mu\right|\in \umml$, if any open neighbourhood of $\left| \mu \right|$ contains $\left| \lambda\right|$, while $\left| \lambda\right| \neq \left| \mu\right|$.  From this follows the notion of \emph{adherence height} of an element $\left| \lambda \right|\in \umml$ which actually coincides with the depth $\mathfrak{d}\!\left( \lambda\right)$. Indeed, because of \cite[Lemma 1]{ohshika}, having  $\left| \lambda \right|$  unilaterally adherent to $\left|\mu\right|$ is the same thing as having $\left| \mu \right| \subset \left| \lambda \right|$, when we consider these latter as geodesic laminations. Thus,  self-homeomorphisms of $\umml$  preserve the depth of associated  strata and therefore they map any element that corresponds to a minimal measured lamination to an element which also corresponds to a minimal measured lamination.  However, at this stage it is not clear why those maps preserve the set $\mathcal{S}$. Indeed, as it is pointed out at the end of Section \ref{sec:strat}, a simple closed geodesic and a minimal measured lamination that fills the interior of a one-holed hyperbolic torus have the same depth. Here is the gap in \cite[Proposition 3.4]{papadop2}. In order to fix it, we  will first prove, using the same idea as in Corollary \ref{cor:disjoint}, that any self-homeomorphism of $\umml$ preserves the disjointness  of supports of elements of $\umml$, and then we will prove that such homeomorphisms also preserve the set of unmeasured laminations of support a minimal measured lamination which is neither a weighted simple closed geodesic nor a maximal.

Let $f$ be a self-homeomorphism of $\umml$. 

Let $\left| \lambda \right|$ and $\left| \mu \right|$ be two elements of $\umml$ with disjoint supports. Without loss of generality one can assume that both $\lambda$ and $\mu$ are minimal measured laminations.  Let $\delta = \lambda+\mu$. It is a measured lamination of depth equal to $1$, so the unmeasured lamination $f\left( \left| \delta \right|\right)$ has an adherence height equal to $1$, meaning that,  its support contains two minimal components. Furthermore, both the supports of $f\left(\left| \lambda \right| \right)$ and $f\left(\left| \mu \right| \right)$ are different and are minimal geodesic laminations,  and  are contained in the support of $f\left( \left| \delta \right|\right)$. Therefore, the support of $f\left( \left| \delta \right|\right)$ must be the union of the support of  $f\left(\left| \lambda \right| \right)$ with the support of $ f\left(\left| \mu \right| \right)$. This shows that  the supports of the latter two unmeasured laminations are disjoint.  

Now, let us follow the strategy of the proof of Corollary \ref{cor:key}. Let $\left| \alpha \right|$ be  a peripheral curve associated with some minimal measured lamination $\nu$. Let $\left| \beta \right|$ be such that $f\left( \left| \alpha \right| \right)\cap \left| \beta \right|$. Then, since $f$ preserves the disjointness of supports we have $\left| \alpha \right| \cap f^{-1}\left( \left| \beta \right|\right)\neq \emptyset$. Thus, $\left| \nu \right| \cap f^{-1}\left( \left| \beta \right|\right)\neq \emptyset$ which implies that $f\left(\left| \nu \right|\right) \cap  \left| \beta \right|\neq \emptyset$. Therefore, by the characterisation of peripheral curves, one can say that $f\left(\left| \alpha \right| \right)$ is a peripheral curve associated with a minimal measured lamination of support $f\left(\left| \nu \right| \right)$.

Following the proof of  Lemma \ref{lemma:key}, we are now able to say that $f$ preserves the set $\mathcal{S}$.

Thus, since $f$ preserves the set of simple closed geodesics and the adherence height, it  induces an automorphisms of the corresponding complex of curves. If $(g,n)=(1,2)$, then $f$ further preserves the separating curves. It is the same proof as in Theorem \ref{main-theo}. Using \cite{ivanov, korkmaz} and \cite{luo}, there exists $g\in \emcg$ such that $g = f$ on $\mathcal{S}$. We can now apply \cite[Lemma 2]{ohshika} to obtain

\begin{corollaire}
The group of self-homeomorphisms of $\umml$ is isomorphic to $\emcg$ if $(g,n)\neq (2,0), (1,2)$, and it is  isomorphic to $\quot{\emcg}{\quot{\mathbb{Z}}{2\mathbb{Z}}}$ if $(g,n)=(2,0)$ or $(g,n)=(1,2)$.
\end{corollaire}

\medskip
\noindent {\bf Acknowledgements.} The author would like to express his sincere gratitude to Athanase Papadopoulos for his suggestion about the notion of stratification. The author also wants to thank Ken'ichi Ohshika for his remark and Vladimir Fock for having pointing out the pathology of the  once-punctured hyperbolic torus.  Furthermore, the author  was  partially supported by the U.S. National Science Foundation grants DMS 1107452, 1107263, 1107367 ``RNMS: GEometric structures And Representation varieties'' (the GEAR Network).

\textsc{Vincent Alberge, Fordham University, Department of Mathematics, 441 East Fordham Road, Bronx, NY 10458, USA}

\textit{E-mail address:} { \href{mailto:valberge@fordham.edu}{\tt valberge@fordham.edu} } 

\end{document}